\documentclass[11pt]{amsart}
\usepackage[english]{babel}
\usepackage{amsfonts,latexsym,amsthm,amssymb,graphicx}
\usepackage{mathabx}
\usepackage[all]{xy}
\usepackage[usenames]{color}
\usepackage{tikz}
\usetikzlibrary{shapes}
\usepackage{xypic}
\usetikzlibrary{decorations.pathreplacing}
\usepackage{amsmath}
\usepackage{graphicx}
\usepackage[enableskew]{youngtab}
\usepackage[utf8]{inputenc}
\usepackage[english]{babel}
\usepackage{tikz}
\usetikzlibrary{arrows}
\usepackage{float}
\usepackage{amscd}
\usepackage{pstricks}
\usepackage{multicol}
\usepackage{blindtext}
\usepackage{geometry}
\usepackage[alphabetic]{amsrefs}
\usepackage{enumerate}
\usepackage{pdfsync}
\usepackage[colorlinks=true, linkcolor=blue, citecolor=blue, urlcolor=blue]{hyperref}
\usepackage{dynkin-diagrams}
\usetikzlibrary{decorations.pathreplacing}

\usepackage{color,soul}
\usepackage{ytableau}

\usepackage{graphicx}

\usepackage{xcolor}
\definecolor{darkgreen}{RGB}{153,255,78}

\newcommand{\Z}{{\mathbb Z}}

\newcommand{\IG}{\mathrm{IG}}

\DeclareMathOperator{\Sp}{Sp}

\newcommand{\IF}{\mathrm{IF}}

\newcommand{\Wodd}{W^{odd}}
\newcommand{\Xev}{X^{ev}}

\newcommand{\Wlarge}{W}

\DeclareMathOperator{\Span}{Span}

\newcommand{\Mb}{\overline{\mathcal{M}}}

\DeclareMathOperator{\ev}{ev}

\newcommand{\QH}{\mathrm{QH}}

\DeclareFontFamily{OT1}{rsfs}{}
\DeclareFontShape{OT1}{rsfs}{n}{it}{<-> rsfs10}{}
\DeclareMathAlphabet{\mathscr}{OT1}{rsfs}{n}{it}

\CompileMatrices

\newtheorem{theo}{Theorem}
\newtheorem{thm}{Theorem}[section]
\newtheorem{lemma}[thm]{Lemma}

\newtheorem{prop}[thm]{Proposition}

{\theoremstyle{definition} \newtheorem{defn}[thm]{Definition}}
{\theoremstyle{remark} 
\newtheorem{example}[thm]{Example}}

\begin{document}

\title[]{On the quantum parameter in the quantum cohomology of a family of odd symplectic partial flag varieties}



\author{Connor Bean}

\address{
Department of Mathematical Sciences,
Henson Science Hall, 
Salisbury University,
Salisbury, MD 21801
}
\email{cbean2@gulls.salisbury.edu}

\author{Caleb Shank}

\email{cshank3@gulls.salisbury.edu}

\author{Ryan M. Shifler}
\email{rmshifler@salisbury.edu}


\begin{abstract}
We will consider a particular family of odd symplectic partial flag varieties denoted by $\IF$.  In the quantum cohomology ring $\QH^*(\IF)$, we will show that $q_1q_2\cdots q_m$ appears $m$ times in the quantum product $\tau_{Div_i} \star \tau_{id}$ when expressed as a sum in terms of the Schubert basis.
\end{abstract}
\maketitle
\section{Introduction}
Let $\IF:=\IF(1,2,\cdots,m;2n+1)$ denote the family of odd symplectic partial flag varieties under consideration. This is the parameterization of sequences $(V_1\subset V_2 \subset \cdots \subset V_m)$, $\dim V_i=i$, of subspaces of $\mathbb{C}^{2n+1}$ that are isotropic with respect to a general skew-symmetric form. The variety $\IF$ contains Schubert varieties $\{X(\lambda): \lambda \in W^{odd} \}$ where $W^{odd}$ is defined in Section \ref{sec:prelim}. See \cite{mihai:odd} for more details on odd symplectic flag varieties.

The quantum cohomology of a smooth variety $X$ is a graded algebra over $\mathbb{Z}[q]$, $q=(q_1,q_2,\cdots,q_t)$, with a $\mathbb{Z}[q]$-basis given by classes in the cohomology ring $H^*(X)$. Multiplication is given by \[ \displaystyle \sigma_i \star \sigma_j=\sum_{d\geq 0; \sigma_k \in H^*(X)} q^dc_{i,j}^{k,d}\sigma_k\] where $c_{i,j}^{k,d}=\left<\sigma_i,\sigma_j,\sigma_k^\vee \right>_d$ is the Gromov-Witten invariant. The degree of $q_i$ is \[ \deg q_i = \int_{{Div_i}^\vee} c_1(T_X)\] where $Div_i$ is the $i$th divisor class and $c_1(T_X)$ is the first Chern class of the tangent bundle of $X$. The study of the quantum cohomology of flag varieties has made progress. For example, Buch and Mihalcea use the technique of curve neighborhoods in \cite{buch.m:nbhds} to produce an equivariant quantum Chevalley formula for any homogeneous variety $G/P$. Limited progress has been made in the study of the quantum cohomology of non-homogenous varieties. For example, the odd symplectic Grassmannian is studied in \cites{pech:quantum,mihalcea.shifler:qhodd}. This manuscript studies a family of odd symplectic partial flag varieties which are non-homoogenous varieties.

The quantum cohomology ring $(\QH^*(\IF),\star)$ is a graded algebra over $\mathbb{Z}[q]=\mathbb{Z}[q_1,\cdots, q_m]$ where $ \deg q_i = 2$ for $1 \leq i \leq m-1$ and $\deg q_m=2(n-m)+3$. The ring has a Schubert basis given by $\{\tau_\lambda:=[X(\lambda)]: \lambda \in W^{odd} \}$. Here we take $\tau_{id}$ to be the class of the Schubert point $pt$ and $\tau_{Div_i}$ to be a divisor class where $1 \leq i \leq m$. The ring multiplication is given by $\tau_\lambda \star \tau_\mu=\sum_{\nu,d} c_{\lambda,\mu}^{\nu,d}q^d \tau_\nu$ where $c_{\lambda,\mu}^{\nu,d}$ is the degree $d$ Gromov-Witten invariant of $\tau_\lambda$, $\tau_\mu$, and the Poicar\'e dual of $\tau_\nu$. We are now ready to state our main result. A more precise statement is given as Theorem \ref{thm:mainthm}.

\begin{theo} \label{thm:orgmain}
Consider the quantum cohomology ring $\QH^*(\IF)$. Then $q_1q_2\cdots q_m$ appears $m$ times in the product $\tau_{Div_i} \star \tau_{id}$ when expressed as a sum in terms of the Schubert basis given by $\{\tau_\lambda: \lambda \in W^{odd} \}$.
\end{theo}

Our strategy will be to use curve neighborhood calculations which we explain next. Let $X$ be a Fano variety. Let $d \in H_2(X,\Z)$ be an effective degree. Recall that the moduli space of genus $0$, degree $d$ stable maps with two marked points $\Mb_{0,2}(X,d)$ is endowed with two evaluation maps $\ev_i \colon \Mb_{0,2}(X,d) \to X$, $i=1,2$ which evaluate stable maps at the $i$-th marked point.

\begin{defn} \label{def:defofcrvnbhd}
Let $\Omega \subset X$ be a closed subvariety. The \emph{curve neighborhood} of $\Omega$ is the subscheme 
\[ 
    \Gamma_d(\Omega) := \ev_2( \ev_1^{-1} \Omega) \subset X
\] 
endowed with the reduced scheme structure.
\end{defn}

The notion of curve neighborhoods is closely related to quantum cohomology. Let $X(\lambda) \subset \IF$ be a Schubert variety, and let $\Gamma_d(X(\lambda)) = \Gamma_1 \cup \Gamma_2 \cup \ldots \cup \Gamma_k$ be the decomposition of the curve neighborhood into irreducible components. By the divisor axiom, any component $\Gamma_i$ of ``expected dimension" will contribute to the quantum product $\tau_{Div_i} \star \tau_{\lambda}$ with $ (\tau_{Div_i},d) \cdot a_i \cdot q^d [\Gamma_i]$, where $a_i$ is the degree of $\ev_2: \ev_1^{-1}(X(\lambda)) \to \Gamma_d(X(\lambda))$ over the given component (see \cite{kontsevich.manin:GW:qc:enumgeom} and Lemma \ref{lem:DivAx}). Therefore the main task is to find the components $\Gamma_i$ of $\Gamma_{(1^m)}(pt)$, where $(1^m)=(\overbracket{1,\cdots,1}^{m})$, that are of expected dimension. That is, the following equation is satisfied: \[\mbox{codim } X(Div_i)+\mbox{codim } pt= \deg q_1q_2\cdots q_m+\mbox{codim } \Gamma_i.\] These components are stated precisely in Proposition \ref{Prop:crvnbd}.

\noindent {\bf Broader Context} Any curve neighborhood of a Schubert variety in the homogeneous space $G/P$ is shown to be irreducible in \cite{buch.m:nbhds}. This limits the number of times that $q^d$ appears for a particular $d \in H_2(G/P,\mathbb{Z})$ in quantum products of Schubert classes. Examples of curve neighborhoods having two irreducible components are given for the odd symplectic Grassmannian in \cites{mihalcea.shifler:qhodd,PechShif:CNBDS}. In particular, in the quantum Chevalley formula for the odd symplectic Grassmannian, $q^1$ appears twice in the quantum product of the divisor class and the class of the point when expressed as a sum in terms of the Schubert basis. The main purpose of this manuscript is to give a specific example where $q^d$ appears a specified number of times as stated in Theorem \ref{thm:orgmain}.

\section{Preliminaries} \label{sec:prelim}
There are many possible ways to index the Schubert varieties of isotropic flag manifolds. Here we recall an indexation using signed permutations. Consider the root system of type $C_{n+1}$ with positive roots
\[
    R^+ = \{ t_i \pm t_j \mid 1 \leq  i < j \leq n+1\} \cup \{ 2 t_i \mid 1 \leq  i \leq n+1 \}
\]
and the subset of simple roots 
\[
    \Delta = \{ \alpha_i := t_i-t_{i+1} \mid 1 \leq i \leq n \} \cup \{ \alpha_{n+1}:= 2t_{n+1} \}.
\] The coroot of $t_i \pm t_j \in R^+$ is $(t_i \pm t_j)^\vee=t_i \pm t_j$ and the coroot of $2t_i \in R^+$ is $(2t_i)^\vee=t_i$. The associated Weyl group $\Wlarge$ is the hyperoctahedral group consisting of \emph{signed permutations}, i.e. permutations $w$ of the elements $\{1, \cdots, n+1,\overline{n+1},\cdots,\overline{1}\}$ satisfying $w(\overline{i})=\overline{w(i)}$ for all $w \in W$. For $1 \leq i \leq n$ denote by $s_i$ the simple reflection corresponding to the root $t_i - t_{i+1}$ and $s_{n+1}$ the simple reflection of $2 t_{n+1}$. In particular, if $1 \leq i \leq n$ then $s_i(i)=i+1$, $s_i(i+1)=i$, and $s_i(j)$ is fixed for all other $j$. Also, $s_{n+1}(n+1)=\overline{n+1}$, $s_{n+1}(\overline{n+1})=n+1$, and $s_{n+1}(j)$ is fixed for all other $j$.

Each subset $ I:=\{ i_1 < \ldots < i_r \} \subset \{ 1, \ldots , n+1 \}$ determines a parabolic subgroup $P:=P_I \subset \Sp_{2n+2}$ with Weyl group $W_{P} = \langle s_i \mid i \neq i_j \rangle$ generated by reflections with indices \emph{not} in $I$. Let $\Delta_P:= \{ \alpha_{i_s} \mid i_s \notin \{ i_1, \ldots , i_r \} \}$ and $R_P^+ := \Span_\Z \Delta_P \cap R^+$; these are the positive roots of $P$. Let $\ell \colon W \to \mathbb{N}$ be the length function and denote by $W^{P}$ the set of minimal length representatives of the cosets in $W/W_{P}$. The length function descends to $W/W_P$ by $\ell(u W_P) = \ell(u')$ where $u' \in W^P$ is the minimal length representative for the coset $u W_P$. We have a natural ordering $1 < 2 < \cdots < n+1 < \overline{n+1} < \cdots < \overline{1}$, which is consistent with our earlier notation $\overline{i} := 2n+3 - i$. 

Let $P$ be the parabolic obtained by excluding the reflections $s_1, s_2, \cdots s_m$. Then the minimal length representatives $W^P$ have the form $(w(1)|w(2)|w(3)|\cdots|w(m)<w(m+1)< \cdots <w(n) \leq n+1)$. Since the last $n+1-m$ labels are determined from the first $m$ labels, we will identify an element in $W^P$ with $(w(1)|w(2)|\cdots|w(m))$. Define $W^{odd}=\{w \in W^P : w(i) < \overline{1} \mbox{ for } 1 \leq i \leq m \}$. 

Let $\Xev:=\IF(1,2,\cdots,m;2n+2)$ be the symplectic partial flag that parameterizes sequences $(V_1\subset V_2 \subset \cdots \subset V_m)$, $\dim V_i=i$, of subspaces of $\mathbb{C}^{2n+2}$ that are isotropic with respect to a skew-symmetric form. Here $P \subset \Sp_{2n+2}$ is the maximal parabolic subgroup corresponding to $I=\{1<2<\cdots<m\}$ and $T_{2n+2}=(t_1,\cdots,t_{n+1},t^{-1}_{n+1},\cdots,t^{-1}_1)$ is a maximal torus for $\Xev$. The Schubert varieties of $\Xev$ are indexed by $\lambda \in W^{P}$ and written as $X(\lambda)$. Since $\IF$ is identified with the Schubert variety $X(\bar{2}\bar{3} \cdots \overline{m} \overline{m+1}) \subset \Xev$, the Schubert varieties of $\IF$ are $\{X(\lambda):\lambda \in W^{odd} \}.$ In addition $\IF$ is smooth. The quantum cohomology ring $\QH^*(\IF)$ has a Schubert basis given by $\{\tau_\lambda:=[X(\lambda)]: \lambda \in W^{odd} \}$. We have that $T=(t_1,\cdots,t_{n+1},t^{-1}_{n+1},\cdots,t^{-1}_2)$ is a maximal torus for $\IF$ and $\dim \IF=m(2n-m+1)$. Next we will give notation to state the Bruhat order.

\begin{example}
Consider $\IF(1,2,3;11)$. This identifies with the Schubert variety $X(\bar{2} \bar{3} \bar{4})$ in $\IF(1,2,3;12)$. Here $(1|\bar{2}|3), (5|\bar{4}|2) \in W^{odd}$ while $(3|\bar{1}|2) \notin W^{odd}.$
\end{example}

\begin{defn}
Let $\lambda, \delta \in W^{odd}$ where $\lambda_i=\lambda(i)$ and $\delta_i=\delta(i)$. Then define the following:
\begin{enumerate}
\item $\Lambda^k:= \left<\Lambda^k_1<\Lambda^k_2<\cdots<\Lambda^k_k \right>$ where $\{\Lambda^k_1,\Lambda^k_2, \cdots, \Lambda^k_k \}=\{\lambda_1,\lambda_2,\cdots,\lambda_k \}$;
\item $\Delta^k:= \left<\Delta^k_1<\Delta^k_2<\cdots<\Delta^k_k \right>$ where $\{\Delta^k_1,\Delta^k_2, \cdots, \Delta^k_k \}=\{\delta_1,\delta_2,\cdots,\delta_k \}$;
\item $\Lambda^k \leq \Delta^k$ if $\Lambda_i^k \leq \Delta_i^k$ for all $1 \leq i \leq k$.
\end{enumerate}
\end{defn}

\begin{lemma}[Bruhat Order \cite{ProctBO}] \label{lem:BO}
Let $\lambda, \delta \in W^P$. Then $\lambda \leq \delta$ if and only if $\Lambda^k \leq \Delta^k$ for all $1 \leq k \leq m$. In particular, if $\lambda,\delta \in W^{odd}$ then $X(\lambda) \subset X(\delta)$ if and only if $\lambda \leq \delta$.
\end{lemma}

\section{The Moment Graph}
Sometimes called the GKM graph, the \emph{moment graph} of a variety with an action of a torus $T$ has a vertex for each $T$-fixed point, and an edge for each $1$-dimensional torus orbit. The description of the moment graphs for flag manifolds is well known, and it can be found in \cite{kumar:kacmoody}*{Ch. XII}. In this section we consider the moment graphs for $\IF$ and  $\Xev$. 

\begin{defn}
The moment graph of $\Xev$ has a vertex for each $w \in W^P$, and an edge $w\rightarrow ws_{\alpha}$ for each 
\[ 
    \alpha \in R^+ \setminus R_{P}^+ = \{ t_i - t_j \mid 1 \leq i \leq m, i< j \leq m+1 \} \cup \{ t_i + t_j, 2 t_i \mid 1 \leq i \leq m, 1 \leq i <j \leq m+1  \}. 
\]
This edge has degree $d=(d_1,d_2,\cdots, d_m)$, where $\alpha^\vee + \Delta_P^\vee = d_1 \alpha_1^\vee +d_2 \alpha_2^\vee + \cdots + d_m \alpha_m^\vee +\Delta_P^\vee$. We will say that $d=(d_1,d_2,\cdots, d_m) \leq d'=(d'_1,d'_2,\cdots, d'_m)$ if $d_i \leq d'_i$ for all $1 \leq i \leq m.$
\end{defn}

\begin{defn}
The moment graph of $\IF$ is the full subgraph of $\Xev$ determined by the vertices $w \in W^{odd}$.
\end{defn}
Next we classify the positive roots by their degree. 
\begin{defn}\label{def:moment-graph-combinat}
Let $(0^a1^b2^c):=(\overbracket{0,\cdots,0}^{a},\overbracket{1,\cdots,1}^{b},\overbracket{2,\cdots,2}^{c})$. Define the following to describe moment graph combinatorics.
\begin{enumerate}
\item Define the following sets which partitions $R^+ \setminus R_{P}^+$.
\begin{enumerate}
\item $R^+_{(0^{i-1}1^{j-i}0^{m-j+1})}=\{t_i-t_j: 1 \leq i <j \leq m\}$;
\item $R^+_{(0^{i-1}1^{m-i+1})}=\{t_i \pm t_j : 1 \leq i \leq j, m < j \leq n+1 \} \cup \{2t_i:1 \leq i \leq m\}$;
\item $R^+_{(0^{i-1}1^{j-i}2^{m-j+1})}=\{t_i+t_j: 1\leq i< j \leq m\}$.
\end{enumerate}
\item A {\it chain of degree $d$} is a path in the (unoriented) moment graph where the sum of edge degrees equals $d$. We will use the notation $uW_P \overset{d}{\rightarrow} vW_P$ to denote such a path.
\end{enumerate}
\end{defn}

In the next lemma we give a formula for the degree $d$ of a chain which is useful to calculate curve neighborhoods. In particular, we will see that the degree of a chain is determined by summing the weights of the edges included in the chain (repetitions are allowed) in the moment graph.

\begin{lemma}[\cite{FW}, Page 8] \label{lem:3.2}
Let $u,v \in W^P$ be connected by a degree $d$ chain
\[
  (u W_P \overset{d} \to vW_P) = (uW_P \to us_{\alpha_1} W_P \to \dots \to us_{\alpha_1}s_{\alpha_2}\ldots s_{\alpha_t}W_P)
\]
where $vW_P=us_{\alpha_1}s_{\alpha_2}\ldots s_{\alpha_t}W_P$ and the $\alpha_j$ are in $ R^+ \backslash R^+_P$. Then $d=(O_1+D_1,O_2+D_2, \cdots, O_m+D_m)$ where
\[ \displaystyle O_i=\sum_{\substack{a \leq i-1 \\ a+b\geq i} } \# \left\{ \alpha_j \in R^+_{(0^a1^b2^{m-a-b})} \right\} \mbox{ and }D_i=2 \cdot \sum_{a+b\leq i-1} \# \left\{ \alpha_j \in R^+_{(0^a1^b2^{m-a-b})} \right\}.\]
\end{lemma}

\section{Proof of main result}

We begin this section by stating Proposition \ref{prop:moment-odd} which gives curve neighborhoods, defined in Definition \ref{def:defofcrvnbhd}, a combinatorial interpretation in terms of the moment graph. Then Lemmas \ref{lem:intsec} and \ref{lem:1inL} demonstrate that $\lambda \in W^{odd}$ is constrained when it is reached by a chain of degree less than or equal to $(1^m)$. This follows with Lemmas \ref{lem:cont} and \ref{lem:length} which gives a precise statement of $\Gamma_{(1^m)}(pt)$ in Proposition \ref{Prop:crvnbd}. Finally, we present our main result in Theorem \ref{thm:mainthm} which follows from Lemma \ref{lem:DivAx}.

\begin{prop}[\cite{buch.m:nbhds}]\label{prop:moment-odd}
Let $\lambda \in \Wodd$. In the moment graph of $\IF$, let $\{v^1, \cdots, v^s\}$ be the maximal vertices (for the Bruhat order) which can be reached from any $u \leq \lambda$ using a chain of degree $d$ or less. Then $\Gamma_d(X(\lambda))=X(v^1) \cup \cdots \cup X(v^s)$.
\end{prop}

\begin{proof}
Let $Z_{\lambda,d}=X(v^1) \cup \cdots \cup X(v^s)$. Let $v:= v^i \in Z_{\lambda,d}$ be one of the maximal $T$-fixed points. By the definition of $v$ and the moment graph there exists a chain of $T$-stable rational curves of degree less than or equal to $d$ joining $u \leq \lambda$ to $v$. It follows that there exists a degree $d$ stable map joining $u \leq \lambda$ to $v$. Therefore $v \in \Gamma_d(X(\lambda))$, thus $X(v) \subset \Gamma_d(X(\lambda))$, and finally $Z_{\lambda,d} \subset \Gamma_d(X(\lambda))$.

For the converse inclusion, let $v \in \Gamma_d(X(\lambda))$ be a $T$-fixed point. By \cite{mare.mihalcea}*{Lemma 5.3} there exists a $T$-stable curve joining a fixed point $u \in X(\lambda)$ to $v$. This curve corresponds to a path of degree $d$ or less from some $u \leq \lambda$ to $v$ in the moment graph of $\IG(k,2n+1)$. By maximality of the $v^i$ it follows that $v \leq v^i$ for some $i$, hence $v  \in X(v^i) \subset Z_{\lambda,d}$, which completes the proof.
\end{proof}

\begin{lemma} \label{lem:intsec}
Let $\mathcal{C}: idW \overset{d} \to \lambda W$ be a chain in the moment graph of $\IF$ where $d \leq (1^m)$. Then we have the inequality \[\left |\Lambda^k \cap \{1,2,\cdots,k \} \right| \geq k-1.\]
\end{lemma}

\begin{proof}
Suppose $\left|\Lambda^k \bigcap \{1,2,\cdots,k \} \right| < k-1$. Then there are at at least two elements $\Lambda^k_{a_1},\Lambda^k_{b_1} \in \Lambda^k$ such that $\Lambda^k_{a_1},\Lambda^k_{b_1} > k$. Since $\Lambda^k_{a_1} > k$ there exists a reflection in the chain $\mathcal{C}$ corresponding to $t_{a_1} - t_{a_2}$ where $a_1 \leq k$ and $a_2 > k$. Also, since $\Lambda^k_{b_1}> k$ there exists a reflection in the chain $\mathcal{C}$ corresponding to $t_{b_1} - t_{b_2}$ where $b_1 \leq k$ and $b_2 > k$. Therefore, $d_k \geq 2$. But $d_k \leq 1$. The result follows.
\end{proof}

\begin{lemma} \label{lem:1inL}
Let $\mathcal{C}: idW \overset{d} \to \lambda W$ be a chain in the moment graph of $\IF$ where $d \leq (1^m)$ and $ \overline{j} \in \{\lambda_1,\lambda_2,\cdots, \lambda_m \}$ for some $2 \leq j \leq m$. The chain $\mathcal{C}$ has a reflection corresponding to the root $2t_j$. In particular, $1 \in \Lambda^j.$ 
\end{lemma}

\begin{proof}
Consider the chain $\mathcal{C}: idW_P \overset{d} \to \lambda W_P$. One of the following three cases must have occurred.
\begin{enumerate}
\item The chain $\mathcal{C}$ has a reflection corresponding to the root $2t_j$;
\item The chain $\mathcal{C}$ has two reflections corresponding to two roots of the form $t_a \pm t_b$ where $a \leq m$ and $b\geq m$;
\item The chain $\mathcal{C}$ has a reflection corresponding to the root $t_a+t_b$ where $a, b\leq m$ and $a<b$.
\end{enumerate}
In the first case we have that \[(2t_j)^\vee=t_j=(t_j-t_{j+1})+(t_{j+1}-t_{j+2})+\cdots+(t_{n-1}-1t_n)+t_n.\] In particular, $d_i \leq 1$ for all $1 \leq i \leq m$. In the second case, the coefficient of $t_m-t_{m+1}$ is 1 when $t_a \pm t_b$ and $t_c \pm t_d$ ($a,c \leq m$ and $b,d \geq m$), are written as a sum of simple roots. Thus, $d_m \geq 2$. This is not possible. In the third case, the coefficient of $t_m-t_{m+1}$ is 2 when $t_a+t_b$ ($a, b\leq m$ and $a < b$) is written as a sum of simple roots. This is not possible. Therefore, the chain $\mathcal{C}$ has a reflection corresponding to the root $2t_j$. Finally, if $1 \notin \Lambda^j$, then $d_j \geq 2$ or $\bar{1}$ appears in $\lambda$. Neither is possible. This completes the proof.
\end{proof}

\begin{lemma} \label{lem:cont}
Let $\mathcal{C}: idW \overset{d} \to \lambda W$ be a chain in the moment graph of $\IF$ such that $d \leq (1^m).$
\begin{enumerate}
\item If $\Lambda_m^m \leq \overline{m+1}$ then $X(\lambda) \subset X(\overline{m+1} \vert 2 \vert 3 \vert \cdots \vert m).$
\item If $\bar{j} \in \{\lambda_1,\lambda_2,\cdots, \lambda_m \}$, where $2 \leq j \leq m$, then \[X(\lambda) \subset X(\overline{j}|2|3|\cdots|j-1|1|j+1|\cdots|m).\]
\end{enumerate}
\end{lemma}

\begin{proof}
We will prove Part (1) first. Let $1 \leq k \leq m$, $\delta=(\overline{m+1} \vert 2 \vert 3 \vert \cdots \vert m)$, and $\Lambda_m^m \leq \overline{m+1}$. It follows that $\Delta^k=(2<3<\cdots<k<\overline{m+1})$. Also, $\left |\Lambda^k \cap \{1,2,\cdots,k \} \right| \in \{k-1,k\}$ by Lemma \ref{lem:intsec}. If $\left |\Lambda^k \cap \{1,2,\cdots,k \} \right|=k$ then clearly $\Lambda^k \leq \Delta^k$. 

Suppose that $\left |\Lambda^k \cap \{1,2,\cdots,k \} \right|=k-1$. Then $\Lambda^k=(1<2<\cdots<\hat{i}<\cdots<k<\lambda_j)$ where $i$ is removed and $\lambda_j \leq \Lambda_m^m \leq \overline{m+1}$. It follows that $\Lambda^k \leq \Delta^k$. Therefore, $\lambda \leq \delta$ and Part (1) follows by Lemma \ref{lem:BO}.

Next we will prove Part (2). Let $1 \leq k \leq m$, $\bar{j} \in \{\lambda_1,\lambda_2,\cdots, \lambda_m \}$, where $2 \leq j \leq m$, and $\delta=(\overline{j}|2|3|\cdots|j-1|1|j+1|\cdots|m)$. There are two cases for $\Delta^k$. 
\begin{enumerate}
\item If $k \leq j-1$ then $\Delta^k=(2<3<\cdots<k<\overline{j})$;
\item if $k \geq j$ then $\Delta^k=(1<2<3\cdots<j-1<j+1< \cdots <k<\overline{j})$. 
\end{enumerate}
If $\left |\Lambda^k \cap \{1,2,\cdots,k \} \right|=k$ then clearly $\Lambda^k \leq \Delta^k$. 

Suppose that $\left |\Lambda^k \cap \{1,2,\cdots,k \} \right|=k-1$. Then $\Lambda^k=(1<2<\cdots<\hat{i}<\cdots<k<\overline{j})$ where $i$ is removed. If $k \leq j-1$ then clearly $\Lambda^k \leq \Delta^k$. If $k \geq j$ then 1 must be included in $\Lambda^k$ by Lemma \ref{lem:1inL}. So, if $k \geq j$, we have that $\Lambda^k \leq \Delta^k$. Therefore, $\lambda \leq \delta$ and Part (2) follows by Lemma \ref{lem:BO}. This concludes the proof. \end{proof}

\begin{lemma} \label{lem:length} We have the following permutation length calculation
\[
\ell(\overline{m+1} \vert 2 \vert 3 \vert \cdots \vert m)=\ell(\overline{j}|2|3|\cdots|j-1|1|j+1|\cdots|m)=2n\] for $2 \leq j \leq m.$ In particular, the union \[X(\overline{m+1} \vert 2 \vert 3 \vert \cdots \vert m) \cup \left(\bigcup_{j=2}^{m} X(\overline{j}|2|3|\cdots|j-1|1|j+1|\cdots|m) \right)\] has $m$ irreducible components of dimension $2n$.
\end{lemma}

\begin{proof}
The lengths $\ell(\overline{m+1} \vert 2 \vert 3 \vert \cdots \vert m)$ and $\ell(\overline{j}|2|3|\cdots|j-1|1|j+1|\cdots|m)$ are calculated by counting the number of simple reflections in a reduced word of the given permutation. 
\end{proof}

\begin{prop} \label{Prop:crvnbd}
Let $n \in \mathbb{Z}^{+}$ and consider $\IF$. Then $ \Gamma_{(1^m)}(pt)$ has $m$ irreducible components of dimension $2n$. Specifically,
\begin{align*}
    \Gamma_{(1^m)}(pt) = &X(\overline{m+1} \vert 2 \vert 3 \vert \cdots \vert m) \cup \left(\bigcup_{j=2}^{m} X(\overline{j}|2|3|\cdots|j-1|1|j+1|\cdots|m)\right).
\end{align*}
\end{prop}

\begin{proof}
This is an immediate consequence of Proposition \ref{prop:moment-odd} and Lemmas \ref{lem:cont} and \ref{lem:length}.
\end{proof}

\begin{lemma}[divisor axiom, \cite{kontsevich.manin:GW:qc:enumgeom}] \label{lem:DivAx}
Let $I_d(\tau_{\lambda},\tau_{\delta},\tau_{Div_i})$ be the 3-point Gromov-Witten Invariant of $\tau_{\lambda}$, $\tau_{\delta}$, and $\tau_{Div_i}$ and $I_d(\tau_{\lambda},\tau_{\delta})$ be the 2-point Gromov-Witten Invariant of $\tau_{\lambda}$ and $\tau_{\delta}$. Then the divisor axiom states \[I_d(\tau_{\lambda},\tau_{\delta},\tau_{Div_i})=(\tau_{Div_i},d)I_d(\tau_{\lambda},\tau_{\delta}).\]

In particular, any component $\Gamma_i$ of $\Gamma_d(X(\lambda)) = \Gamma_1 \cup \Gamma_2 \cup \ldots \cup \Gamma_k$ that satisfies \[\mbox{codim } X(Div_i)+\mbox{codim }pt= deg \mbox{ }q^{(1^m)}+\mbox{codim } \Gamma_i\] will contribute to the quantum product $\tau_{Div_i} \star \tau_{\lambda}$ with $ (\tau_{Div_i},d) \cdot a_i \cdot q^d [\Gamma_i]$, where $a_i$ is the degree of $\ev_2: \ev_1^{-1}(X(\lambda)) \to \Gamma_d(X(\lambda))$ over the given component.
\end{lemma}

\begin{thm}\label{thm:mainthm}
In the quantum cohomology ring $\QH^*(\IF)$ we have that \[\tau_{Div_i} \star \tau_{id}=(\tau_{Div_i},d)q_1q_2\cdots q_m \left(a_1\tau_{(\overline{m+1} \vert 2 \vert 3 \vert \cdots \vert m)}+ \sum_{j=2}^m a_j\tau_{(\overline{j}|2|3|\cdots|j-1|1|j+1|\cdots|m)}\right)+\mbox{other terms} \]
where $a_j$ is the degree of $\ev_2: \ev_1^{-1}(pt) \to \Gamma_d(X(\lambda))$ over $X(\overline{m+1} \vert 2 \vert 3 \vert \cdots \vert m)$ when $j=1$ and $X(\overline{j}|2|3|\cdots|j-1|1|j+1|\cdots|m)$ when $2 \leq j \leq m$.
\end{thm}

\begin{proof}
First notice that each irreducible component of $\Gamma_{(1^m)}(id)$ is of expected dimension. That is, $\mbox{codim } X(Div_i)+\mbox{codim }pt= \deg q^{(1^m)}+\left(\dim \IF -2n\right)$. The result follows by the divisor axiom.
\end{proof}

\bibliographystyle{halpha}
\bibliography{bibliography}
\end{document}